\documentclass[10pt]{article}
\usepackage[ansinew]{inputenc}

\usepackage{latexsym,bezier,enumerate,longtable,dcolumn,color,pstcol,xspace,curves,mathpazo,url}
\usepackage{amsmath,amsthm,amsopn,amstext,amscd,amsfonts,amssymb,fancybox,pst-node,pst-tree}

\usepackage{epsfig,epic,graphicx,pstricks}

\mathchardef\isinpunto="0010
\newcommand{\ds}{\displaystyle}

\newcommand{\real}{I\!\! R}

\newcommand{\B}[1]{\boldsymbol{#1}}

\newcommand{\abs}[1]{\lvert#1\rvert}

\def\blok{\hspace*{\fill}{$\Box$}\vspace*{0.5cm}}

\theoremstyle{plain}
\newtheorem{proposition}{\bf Proposition}
\newtheorem{theorem}{\bf Theorem}

\theoremstyle{definition}
\newtheorem{definition}{\bf Definition}
\newtheorem{example}{\bf Example}
\newtheorem{remark}{\bf Remark}

\title{\baselineskip .2in
\bf The Shapley Group Value\footnote{This research has been
supported by I+D+i research project  MTM2011-27892 from the
Government of Spain.}}

\begin{document}
\setlength{\abovedisplayskip}{5mm} 
\setlength{\abovedisplayshortskip}{3mm} 
\setlength{\belowdisplayskip}{6mm} 
\setlength{\belowdisplayshortskip}{5mm} 

\maketitle \vspace*{-1cm} {\baselineskip .2in
\begin{center}
{\small
\begin{tabular}{l} \bf Ram\'on Flores
\\ Departamento de Matem\'aticas, Universidad Aut\'onoma de Madrid,  Spain
\\ e-mail: \url{ramon.flores@uam.es}
\\[2mm]
\bf Elisenda Molina
\\ Departamento de Estad\'{\i}stica, Universidad Carlos III de Madrid, Spain
\\ e-mail: \url{elisenda.molina@uc3m.es}
\\[2mm]
\bf Juan Tejada
\\ Instituto de Matem\'atica Interdisciplinar (IMI), Departamento de Estad\'{\i}stica e Investigaci\'on Operativa,
\\ Universidad Complutense de Madrid, Spain
\\  e-mail: \url{jtejada@mat.ucm.es}
\end{tabular}}
\end{center}

\begin{abstract}
Following the original interpretation of the Shapley value (Shapley, 1953a) as a priori evaluation of the prospects of a player in a
multi-person interaction situation, we propose a {\em group value}, which we call {\em the Shapley group value}, as a priori evaluation of the prospects of a group of players in a coalitional game when acting as a unit. We study its properties and we give an axiomatic characterization. Relaying on this valuation we analyze the profitability of a group. We motivate our proposal by means of some relevant applications of the Shapley group value, when it is used as an objective function by a decision maker who is trying to identify an optimal group of agents in a framework in which agents interact and the attained benefit can be modeled by means of a transferable utility game. As an illustrative example we analyze the problem of identifying the set of key agents in a terrorist network.

\end{abstract}

\section{Introduction}

One of the most important one-point solution concepts in the framework of coalitional games with side payments is the Shapley value (Shapley, 1953a), which proposes an allocation of resources in multiperson interactions taking into account the power of players in their various cooperation opportunities. Since the pioneering work of Lloyd S. Shapley, many different values have been proposed, as for example the nucleolus (Schmeidler, 1969);  the $\tau$-value (Tijs, 1981) or the least square values (Ruiz, Valenciano and Zarzuelo, 1998). There is also a vast literature focused on extensions, modifications and generalizations of the Shapley value: weighted Shapley values (Shapley, 1953b); semi-values (Dubey, Neyman and Weber, 1981); the value for large games (Aumann and Shapley, 1974); values for NTU games which generalize the Nash bargaining solution (Nash, 1950) and the Shapley value (Harsanyi, 1963; Shapley, 1969; Maschler and Owen, 1992); the Myerson value for games with graph restricted communication (Myerson, 1977), values for games with coalition structures (Aumann and Dr\`{e}ze, 1974; Owen, 1977; Hart and Kurz, 1983), and others.

In this paper, we propose an extension of the Shapley value inspired in the question originally addressed by Shapley in his seminal paper (see Hart 1987): {\em  How would one evaluate the prospects of a player in a multiperson interaction, that is, in a game?} Following this interpretation, we propose a {\em group value} as a priori evaluation of the prospects of a group of players in a multi-person game when acting as a group, which takes into account the power of groups in their various cooperation opportunities without imposing on the other agents any concrete coalition structure. Mathematically, our proposal is related with the Shapley value of certain quotient games (Owen, 1977), also called merging games by Derks and Tijs (2000), which capture the situation when all the members of a group commit themselves to bargain with the others as a unit. It is remarkable however that they do not deal with the problem of assigning values to groups.

A key observation in our proposal is that we do not need to suppose necessarily that the players know each other nor agree to act jointly; instead, we assume the existence of an {\emph external agent}, the decision maker, that is able to coordinate the actions of the members of the group. This is the case for instance of terrorist organizations like Al Qaeda, or secret societies; in which there exists a leader (or a set of leaders) who sends a common signal that all the agents in the group are willing to follow. In this work we describe alternative situations in which this type of external coordination occurs.

Following Roth (1977), the ultimate aim of our proposal is to serve as an utility function for a decision maker who is trying to find an optimal group under certain conditions - for instance, of a given size- in those situations in which agents are immersed in a cooperative game. To be specific, we illustrate its applicability in three different settings which share two relevant features: $(i)$ the objective is the selection of an optimal group, rather than the best individual; and $(ii)$ the performance of a group depends on its interaction with the rest of agents. In this context, to maximize the characteristic function entails a too restrictive assumption over the rest of agent's behavior, i.e., only one scenario (mainly, the worst one) is evaluated. On the contrary, we show that maximizing the {\em value of a group} allows us to consider a more general setting, in which more than one scenario concerning other agents' actions is taken into account.

Section \ref{motivating-examples} is devoted to a general presentation of the problem we deal with. We first introduce some standard concepts and notation on Game Theory that will be used throughout this paper, and then we describe three different cases in which the need for a group valuation arises. In Section \ref{def-axiomatizacion} we introduce the notion of {\em group value} and we define our proposal, which we call the {\em Shapley group value}. We give an axiomatic characterization for it and we analyze its properties. In Section \ref{profitability} we analyze the profitability of a group by means of comparing its valuation as a group with the sum of its members' individual Shapley values.
 In Section \ref{ejemplos} we explore one of its potential applications previously considered in Section \ref{motivating-examples}.  Section \ref{conclusions} concludes the paper.

{\bf Acknowledgements.} We would like to warmly thank Javier Castro (Universidad Complutense de Madrid) for the simulation program used to obtain the numerical results in Section \ref{ejemplos}.

\section{Motivation and notation}
\label{motivating-examples}

A cooperative game in coalitional form with side payments, or with
 transferable utility, is an ordered pair $(N,v)$, where
 $N$ is a finite set of players and $v:2^N\rightarrow \real$, with $2^N=\{ S \,\vert \, S\subset N\}$, is a {\em characteristic function} on $N$ satisfying $v(\emptyset)=0$. For any coalition $S\subset N$,  $v(S)\in \real$ is the {\em worth} of coalition $S$ and  represents the reward that coalition $S$ can achieve by itself if all its members act together. Since we will restrict to the case of TU games in the sequel, we will refer to them simply as {\em games}. For brevity, throughout the paper, the cardinality of sets (coalitions) $N, S$ and $C$ will be denoted by appropriate small letters $n,s$ and $c$, respectively. Also, for notational convenience, we will write singleton $\{ i\}$ as $i$, when no ambiguity appears.

Let ${\cal U}=\{1,2,\dots\}$ be the {\em universe of players}, and let ${\cal N}$ be the class of all non-empty finite subsets of ${\cal U}$. For an  element $N$ in ${\cal N}$, let ${\cal G}_N$ denote the set of all characteristic functions on player set $N$, and let ${\cal G}=\displaystyle\bigcup_{N\in {\cal N}} {\cal G}_N$ be the set of all characteristic functions.\footnote{We will use interchangeably the two terminologies, game and characteristic function, when no ambiguity appears.} A {\em value} $\varphi$ is an assignation which associates to each game $v\in {\cal G}_N$, $N\in {\cal N}$, a real vector $\varphi(N,v)=(\varphi_i(N,v))_{i\in N}\in \mathbb{R}^N$, where $\varphi_i(N,v)\in \mathbb{R}$ represents the {\em value} of player $i$, $i\in N$. Shapley (1953a) defines his value as follows:
\begin{equation}
\displaystyle
 \phi_i (N,v) = \sum_{S\subset N\setminus i} \frac{s!(n-s-1)!}{n!} \bigl( v(S\cup \{ i\}) -
v(S)
 \bigr ), \quad i\in N.
\label{shapley-marginal}
\end{equation}
 The {\em value} $\phi_i (N,v)$ of each player, which is a weighted average of his marginal contributions, admits different interpretations, such as the {\em payoff} that player $i$ receives when the Shapley value is used to predict the allocation of resources in multiperson interactions, or his {\em power} when averages are used to aggregate the power of players in their various cooperation opportunities.  As announced in the introduction, our proposal for measuring the value of a group is based on the question originally addressed by Shapley in his seminal paper: we interpret the value as the {\em expectations} of a player (group) in a game $(N,v)$. In other words, $\phi_i(N,v)\in \mathbb{R}$ is an {\em a priori} value that measures the prospects of player $i\in N$ in the game $v\in {\cal G}_N$, and can be used as an objective function for selecting {\em key players}.\footnote{Note that in general $v(i)\neq \phi_i(N,v)$.}  The approach just discussed is undertaken in the next three cases, already considered in the literature:

\begin{enumerate}[$(i)$]
\item
In Lindelauf, Hamers and Husslage (2013), the authors introduce a game-theoretic approach to identify the key players in a terrorist network. They considered four different weighted extensions of the {\em connectivity game} (Amer and Gimenez, 2004) to capture the structure of the terrorist organization as well as additional individual information about the terrorists, and then they proposed to calculate the Shapley value of each game in order to identify the key players. In Section \ref{ejemplos}, where we analyze in detail this application, we recover the formal definitions of those games.
\item
In Narayanam and Narahari (2011), the authors also introduce a game-theoretic approach to address the {\em target set selection problem} in the framework of diffusion of information. They consider the {\em linear threshold model} (Schelling (1971), Granovetter (1978), and
 Kempe, Kleinberg and Tardos (2005)) to model the role of the social structure in the sharing of information and the formation of opinions dynamics. The propagation process in this case is characterized by  a set of agents $N$, which are connected through a social network. The flow of available information is captured by a weight matrix $\B{W}$ whose entries are understood as \emph{influence weights}; in particular,  $w_{ij}$ quantifies the weight that agent $i$ assigns to agent $j$. It is assumed that these weights are normalized in such a way that $\sum_{j\in N_i} w_{ij} \leq 1$, where $N_i$ represents the set of neighbors of agent $i$, for all $i\in N$. In this model, it is also assumed that each agent has two possible states: {\em active}, if he has adopted the information that is being propagated, and {\em inactive} otherwise. From a dynamic point of view, it is assumed that the status of the agents may change as time goes by. At each date, agents communicate with their neighbors in the social network and update their state. The updating process is simple: all agents that were active in step $(t-1)$ remain active at step $t$; and every inactive agent at step $(t-1)$ becomes active if the sum of the weights of his active neighbors' from the previous period is at least $\theta_i$, a threshold which represents the weighted fraction of the neighbors of $i$ that must become active in order to activate agent $i$.

The authors define a game $(N,v)$, where the value of coalition $S$ is defined to be the expected number of active nodes at the end of the diffusion process when initially all agents in $S$ are active, whereas all agents in $N\setminus S$ are inactive, and assuming all thresholds $\theta_i$ are chosen uniformly at random from the interval $[0,1]$ initially. Then, they propose to calculate the Shapley value of the game in order to rank the agents. Taking into account that the $k$ agents with highest Shapley value are not in general the optimal set of $k$ agents, they propose an heuristic procedure, based on the Shapley value of each agent and the social network structure, to select the  {\em key set} of $k$ agents.
\item
In Conklin, Powaga and Lipovetski (2004) the authors introduce a game-theoretic approach for the identification of sets of key drivers in customer satisfaction analysis, on the basis of Kano's theory (Kano {\it et al.}, 1984) of the relationship between product quality and customer satisfaction, and the information given by a random sample of customers. The attributes take the role of the players, and the characteristic function of the game measures the ability of every group of attributes to predict the dissatisfaction level of customers. To be specific,
$$
v(S)=\underbrace{P\{ \sum_{i\in S} M_i>0 \, /\, D=1\}}_{reach} - \underbrace{P\{ \sum_{i\in S} M_i>0 \, /\, D=0\}}_{noise},
$$
where $M_i\in\{0,1\}$ takes the value 1 if attribute $i$ has failed, and 0 otherwise; and $D\in\{ 0,1\}$ reflects the overall dissatisfaction ($D=1$) with the product. The above probabilities are estimated through the opinions of a random sample of customers as proportions of the failed within those who are dissatisfied ({\em reach}) and non-dissatisfied ({\em noise}). Then, the authors propose to use the Shapley value for ordering the attributes. Then, they add attributes to the list of key dissatisfiers following that order until a point where the added noise overwhelms the added reach. This problem is in fact a particular instance of a  multicriteria decision problem, in which the aim is to rank or score alternatives
according to several (often conflicting) points of view, called criteria (see Grabisch {\it et al.}, 2002). In this context Pint\'er (2011) uses the Shapley value for selecting the predictors in a linear regression model.
\end{enumerate}

Note that in the examples considered above, there exists in fact an {\em external decision maker} who is interested in finding an optimal {\em group} of agents, rather that an optimal agent:

\begin{enumerate}[$(i)$]
\item
 In the first example the police wants to identify a small group of terrorists to neutralize in order to break up the criminal organization. Or, it could be the case, that they were interested in selecting a small group of criminals to mislead in order to optimally diffuse their own information through the network (by using them as seed). In the first case, the goal could be to find the group of terrorists of a given size $1\leq k <n$ whose removing turns into a maximum reduction of the criminal activity. In the second case, the goal could be to find a group of minimum size that achieves a given percentage of information spreading. In both cases, the police needs an {\em objective function} which evaluates the a priori performance of every group in each of the two settings.
\item
In the {\em target set selection problem} of example two, the goal is to find a set of $k$ key agents that would maximize the spreading of information through the network. Thus, we need again an objective function which measures the a priori ability of each group to spread information.
\item
In Conklin, Powaga and Lipovetski (2004), where the authors' goal is  the identification of {\em sets} of $k$ key drivers in customer satisfaction analysis, the need for a objective function which gives a priori valuation of each group of attributes to predict the dissatisfaction level of customers is also clear. Again the problem is to find the group of $k$ drivers which maximizes this objective function.
\end{enumerate}

It must also be remarked that, analogously to what happens for the individual case, in which $v(i)$ is not a proper valuation of the performance of player $i\in N$ in the game, a direct use of $v(S)$ to measure the a priori value of group $S$ is not in general the best approach to solve this problem. For instance, to maximize $v(\cdot)$ in the second example implies a pessimistic scenario in which none of the agents out of coalition $S$ whose diffusion power is being evaluated adopts the product spontaneously. The same argument remains valid for the other two situations considered above.

Thus, since measuring the expected value of a group is a relevant question, and taking into account that the $k$ more valuable agents (from the individual point of view) do not form in general the most valuable group of $k$ agents, the need for a specific group valuation is clear.

\section{The Shapley group value: definition and axiomatic characterization}
\label{def-axiomatizacion}

We aim to define an appropriate {\em group valuation} which can be used as an utility function for an {\em external decision agent} who is trying to select a group of players according to some specified goal. In this section we define our proposal, we study its properties and we characterize it axiomatically.

In this setting, a \emph{group value} $\xi^g$ is a mapping that to any non-empty and finite player set $N\in {\cal N}$, any game $v\in {\cal G}_N$, and any group $C$ with $C\subset N\in {\cal N}$, associates a real number $\xi^g(C;N,v)$, satisfying $\xi^g (\emptyset;N,v) =0$. $\xi^g(C;N,v)\in\real$ is the {\em value} of group $C$ and measures the prospects of group $C$ in $v$ when players in $C$ act as a unit.

The properties of a characteristic function hold for $\xi^g(\cdot;N,v)$, but the interpretation we give of every value $\xi^g(C;N,v)$, $C\subset N$, does not coincide with the usual interpretation of the worth $v(C)$ of coalition $C$ in a game $v\in {\cal G}_N$. As it happens in the case of the Shapley value of the game, the group value of a game determines a different utility function for games which is, in words of Roth (1977), ``compatible with the existing utility function that defines the game''.
   In fact, following the Shapley value approach, we propose to evaluate the performance of a group averaging its marginal contribution over all possible scenarios. Going back to example two, when evaluating the diffusion power of group $S$ we take into account all feasible scenarios in which some of the remaining agents adopt the product spontaneously.

First, in order to define our proposal we must consider what group integration means for the applications we have in mind. In this framework, group integration does not necessarily imply that agents in $C$ make an agreement to act jointly.  For instance, going back to the diffusion of information case, there exists a external agent who can activate the nodes that are used as seeds to diffuse the innovation through the network, and the activated nodes are not in general aware about the other selected seeds' identities. The same occurs when the police selects a group of terrorists to turn back into double agents, or to misinform in order to spread their misinformation through the criminal organization network. Therefore, when measuring group $C$'s expectations we will evaluate them like a {\em unit} anyway, and we will adopt the {\em merging of players} approach of Derks and Tijs (2000), who analyze the {\em profitability} of group formation in a more general setting (but not, as said above, the problem of assigning values to groups). In that case all the agents of $C$ are replaced by a single player $\mathbf{c}$, who can act as a proxy of any agent in $C$.

Formally, let $v\in {\cal G}_N$, and let $C\subset N\in {\cal N}$ be any non-empty coalition. Now, let us consider the so called {\em $C$-partition}, denoted by  $\mathcal{P}_C$, consisting of the compartments $C$ and the one-person coalitions of players outside $C$. Then,  the {\em merging game} with respect to $\mathcal{P}_C$ is the $(n-c+1)$-person cooperative game $(N_C,v_C)$, where the
agent set $N_C=(N\setminus C)\cup \{ \mathbf{c}\}$ with $\mathbf{c}$ as a single proxy player $\mathbf{c}\equiv C$, and $v_C$ is of the form:
\begin{equation}
v_C(S)=\begin{cases} v(S), & \text{ if $\mathbf{c}\notin S$,}
\\
v(S\cup C) & \text{ if $\mathbf{c}\in S$,}
\end{cases} \quad \forall \; S \subset N_C.
\label{merging-game}
\end{equation}

\begin{definition}
The {\em Shapley group value} is the group value that assigns for every $v\in{\cal G}_N$, $N\in {\cal N}$, the valuation mapping $\phi^g(\cdot; N,v)$ given by:
$$
\phi^g(C;N,v)=\phi_{\mathbf{c}} (N_C,v_C), \text{ for each group $\emptyset\neq C\subset N$},
$$
where $(N_C,v_C)$ is the merging game with respect to $C$.
\label{group-Shapley-def}
\end{definition}

Since each $\phi^g(C;N,v)$ is obtained by applying the Shapley value to a merging game, it is remarkable (and somewhat obvious)  that for every coalition with at least two players the corresponding merging game is different. One-person coalitions $C_1=\{i\}$ and $C_2=\{j\}$ are the unique cases in which the two merging games,  $(N_{C_1},v_{C_1})$ and $(N_{C_2},v_{C_2})$, are the same for two different groups $C_1$ and $C_2$. Trivially these two merging games coincide with $(N,v)$.

Our goal is to find an axiomatic characterization of the defined Shapley group value. We first propose and analyze some relevant properties of a group value, which extend to this setting the main properties involved in the many axiomatic approaches provided for the Shapley value.

We first recall some definitions. A game $(N,v)$ is a \emph{unanimity game} if there exists a coalition $S\subset N$ such that for every $T\subset N$, $v(T)=1$ if $S\subset T$, and $v(T)=0$ otherwise. In this case, the game is usually denoted by $(N,u_S)$. Unanimity games are a basis of the vector space ${\cal G}_N$ of all games with the player set $N$. Also recall that given $v\in {\cal G}_N$, $i\in N$ is a {\em dummy player} if $v(S\cup i)=v(S)+v(i)$ for all $S\subset N$. A dummy player with $v(i)=0$ is said to be a {\em null player} in $v$. A game $v\in {\cal G}_N$ is {\em monotonic} if $v(T)\geq v(S)$ for all $S\subset T\subset N$.

\subsection*{Properties}

Let $\xi^g$ be a group value defined over the set of all games ${\cal G}$. Then, $\xi^g$ verifies:
\begin{enumerate}[($P1$)]
\item
{\it  $G$-coherence}, if $\sum_{i\in N}\xi^g( i;N,v)=\xi^g(N;N,v)$ for all $N\in {\cal N}$ and $v\in {\cal G}_N$;
\item
{\it $G$-dummy player}, if $\xi^g({C\cup i};N,v)=\xi^g ({C};N,v)+v(i)$ for all $C\subsetneq N\in {\cal N}$, $i\in N\setminus C$ and $v\in {\cal G}_N$, whenever  $i$ is a dummy player;
\item
{\it $G$-null player}, if $\xi^g({C\cup i};N,v)=\xi^g(C;N,v)$ for all $C\subsetneq N\in {\cal N}$, $i\in N\setminus C$ and $v\in {\cal G}_N$, whenever  $i$ is a null player;
\item
{\it $G$-anonymity}, if for all $C\subset N\in {\cal N}$ and for all permutations $\pi$ of the player set $N$, $\xi^g({\pi(C)};\pi(N), \pi v)=\xi^g(C;N, v)$, where $\pi v(S):= v(\pi(S))$, and being $\pi(S)=\{ \pi(i)\ | i\in S\}$;
\item
{\it $G$-linearity}, if $\xi^g(C;N,\alpha_1 v+ \alpha_2 w)=\alpha_1\xi^g(C;N,v)+\alpha_2\xi^g(C;N,w)$ for all $C\subset N\in {\cal N}$,
 $\alpha_1,\alpha_2\in \mathbb{R}$, and games  $v,w\in {\cal G}_N$,  where $\alpha_1 v+ \alpha_2 w \in {\cal G}_N$ is given by $(\alpha_1 v+ \alpha_2 w)(S)=\alpha_1v(S)+\alpha_2w(S)$ for all $S\subset N$;
\item
{\it $G$-coalitional balanced contributions} (or $G$-CBC for short), if for all $C\subsetneq N \in {\cal N}$, $i,j\in N\setminus C$ and $v\in {\cal G}_N$, we have
\begin{multline}
(\xi^g({C\cup i};N,v)-\xi^g(C;N,v))-(\xi^g({C\cup i};N\backslash j,v_{-j})-\xi^g(C;N\backslash j,v_{-j})) =
\\[3mm]
(\xi^g({C\cup j};N,v)-\xi^g(C;N,v))-(\xi^g({C\cup j};N\backslash i,v_{-i})-\xi^g(C;N\backslash i,v_{-i})),
\label{CBC-larga}
\end{multline}
where $v_{-i}\in {\cal G}_{N\setminus i}$ stands for the restriction of the characteristic function $v$ to the set of players $N\backslash i$;
\item
{\it $G$-symmetry over pure bargaining games} (or $G$-SPB for short), if $\xi^g(C;N,u_N)=\frac{1}{n-c+1}$ for each non-empty $C\subset N\in {\cal N}$, where $(N,u_N)$ is the unanimity game with respect to the grand coalition;
\item
{\it $G$-coalitional monotonicity}, if $\xi^g(C;N,v)\leq \xi^g(C;N,w)$ for all $C\subset T \subset N\in {\cal N}$ and for all games $v,w\in {\cal G}_N$ such that $v(S)= w(S)$ for all $S\neq T$ and $v(T)<w(T)$;
\item
{\it $G$-strong monotonicity}, if  $\xi^g(C;N,v)\leq \xi^g(C;N,w)$ for all $C\subset N\in {\cal N}$, and for all games $v,w\in {\cal G}_N$ for which $v(S\cup C)-v(S)\leq w(S\cup C)-w(S)$ for all $S\subset N\setminus C$.
\item
{\it $G$-positivity}, if $\xi^g(C;N,v)\geq 0$ for all $C\subset N\in {\cal N}$ and any monotonic game $v\in {\cal G}_N$;
\item
{\it $G$-relative invariance with respect to strategic equivalence}, if $\xi^g(C;N,w)=a\xi^g(C;N,v)+\sum_{i\in S} b_i$ for
 all $C\subset N\in {\cal N}$, $v\in G_N$, $a>0$ and $b\in \mathbb{R}^n$, where $w\in G_N$ is given by $w(S)=av(S)+\sum_{i\in S} b_i$ for all $S\subset N$;
\item
{\it $G$-coalitional strategic equivalence}, if $\xi^g(C;N,v)=\xi^g(C;N,v+\lambda u_T)$ for all $C\subsetneq N\in {\cal N}$, $\emptyset \neq T\subset N\setminus C$, $\lambda\in \mathbb{R}$ and $v\in {\cal G}_N$;
\item
{\it $G$-fair ranking}, if for all $T\subset N\in {\cal N}$ and every pair of games $v,w\in {\cal G}_N$ such that $v(S)=w(S)$  for all $S\neq T$, $\xi^g ({C_1};N,v) > \xi^g ({C_2};N,v)$ implies $\xi^g ({C_1};N,w) > \xi^g ({C_2};N,w)$ for all $C_1,C_2 \subset T$ with $\abs{C_1}=\abs{C_2}$.
\end{enumerate}

Regarding the fact that $(N,\xi^g(\cdot;N,v))$ could be formally considered a game in ${\cal G}_N$, some of the previous properties can be interpreted in terms of games. For instance, properties $P2$ and $P3$ can be restated as ``if $i$ is a dummy (null) player in $(N,v)$, then it is a dummy (null) player in $(N,\xi^g(\cdot;N,v))$''.

$G$-coherence is inspired in the coherence property considered  by Calvo and Guti\'errez (2010) in the setting of games with coalition structures. Coherence, which was called {\em coalitional structure equivalence}  in Albizuri (2008) (see also Hart and Kurz, 1983), requires the value to give the same both when all the players are joined and when nobody is joined.

$G$-dummy, $G$-null, $G$-anonymity, $G$-linearity, and $G$-relative invariance with respect to strategic equivalence generalize their counterparts for individual values for general games. Note that linearity of the group value is closely related to the following property, which is consistent with the interpretation of the group value as the expected utility of playing a game: ``the group value of a probabilistic mixture of two games (i.e., with probability $p$ the game $v$ is played, and with probability $p'=1-p$, the game $w$ is played) is the same mixture of the group values of the two games'' (Hart and Kurz, 1983).

$G$-coalitional strategic equivalence, which generalizes $G$-relative invariance with respect to strategic equivalence to coalitions $T\subset N$ with $\abs{T}\geq 2$, adapts the same property of individual values, introduced by Chun (1989) to characterize the Shapley value by adding triviality (the value of a null game is the null vector), efficiency and fair ranking -also defined by Chun (1989), and that can be generalized to $G$-fair ranking-.

$G$-coalitional monotonicity and $G$-strong monotonicity generalize the corresponding monotonicity properties considered by Young (1985), who characterizes the Shapley value by means of efficiency, symmetry and strong monotonicity. Moreover, $G$-positivity extends positivity, introduced by Kalai and Samet (1987).

$G$-coalitional balanced contributions, which together with $G$-symmetry over pure bargaining games, plays a crucial role in the axiomatic characterization of the Shapley group value, generalizes the balanced contribution property which adding efficiency characterizes the Shapley value (Myerson, 1977). $G$-CBC states that for any group $C\subset N\setminus \{ i,j\}$, the impact of player $j$'s presence over the marginal contribution of player $i$ to the value of group $C$ equals the impact of  player $i$'s presence over the marginal contribution of player $j$ to the value of the same group $C$. This idea generalizes the original balanced contribution property of Myerson: ``{\em $\dots$ that player $i$ contribute to player $j$'s payoff what player $j$ contributes to player $i$'s payoff.}'' (Winter, 2002).  Note that in this case payoffs refer to players' values, and  therefore the original property is a particular case of $G$-CBC for $C=\emptyset$.

$G$-SPB  is related to the {\it neutrality to strategic risk} property considered by Roth (1977) to characterize the Shapley value as a utility function. Let $\succsim$ be a preference relation over games and positions, where an element $(v,i)\in {\cal G}_N\times N$ represents the position $i\in N$ in the game $v\in {\cal G}_N$ that the decision maker is evaluating (and comparing with other positions and other games). In this setting, Roth distinguishes between two kinds of risks: {\em ordinary risk}, involving the uncertainty which arises from the chance mechanism involved in lotteries; and {\em strategic risk}, that involves the uncertainty which arises from the interaction in a game of the strategic players.\footnote{Roth (1977) only considers as strategic players those who are not null.} Roth (1977) states that a preference relation $\succsim$ is {\em neutral to strategic risk} if $(u_S,i)\sim (\frac{1}{s}u_{\{i\}},i)$, where $\sim$ stands for the {\em indifference relation}, for all unanimity games $u_S$ (seen as pure bargaining games). That is, the {\em certainty equivalent} of playing $u_S$ in position $i$ should be to receive $\frac{1}{s}$ for sure. In our case, $G$-SPB can be seen as a restriction of the original {\it neutrality to strategic risk} property in the sense that it is only imposed for pure bargaining games over the grand coalition $N$ (see Remark \ref{sobre-axiomas} after Theorem \ref{caracterizacion}). It is worthy to remark that $G$-SPB leads to regard each group as one representative, independent of the number of original players it is composed of, when all players are strictly necessary. In a voting game in which the vote of all players is needed in order to pass a bill, all of them are equally powerful regardless of the number of seats they originally have.

\begin{theorem}
The Shapley group value $\phi^g$ satisfies all properties $P1$ to $P13$.
\end{theorem}

\begin{proof}
Let us arbitrarily fix the sets $C$ and $N$, $C\subset N\in {\cal N}$, and let $v\in {\cal G}_N$. Since for every $i\in N$, $\phi^g(i;N,v)$ coincides with the Shapley value $\phi_i(N,v)$ and $\phi^g(N;v,N):=\phi_{\mathbf{n}} (N_N,v_N)$, which equals $v(N)$, property $G$-coherence $P1$ is a consequence of the Shapley value's efficiency.

To check properties $P2$ to $P5$, and $P8$ to $P13$, we use the following common scheme. $i)$ To note that  $\phi^g ({C};N,v)$ is the Shapley value of the proxy player $\mathbf{c}$ in $(N_{C},v_{C})$; $ii)$ to analyze the implications of the considered property in terms of the merging game $(N_{C},v_{C})$ and its (individual) Shapley value; and $iii)$ to rely on the Shapley value properties to conclude the proof. For instance, the $G$-null player property is clear because $i)$ $\phi^g ({C\cup i};N,v)$ is the Shapley value of the proxy player $\mathbf{c'}\equiv C\cup i$ in $(N_{C\cup i},v_{C\cup i})$; $ii)$ since all the marginal contributions of $i$ are zero, the merging game $(N_{C\cup i},v_{C\cup i})$ coincides with the game $(N_C\setminus i, v_C\vert_{N_C\setminus i})$ identifying $\mathbf{c'}\equiv \mathbf{c}$. Note also that $i$ is a null player in $(N_C,v_C)$. Then, $iii)$ $\phi^g ({C\cup i};N,v):=\phi_{\mathbf{c'}} (N_{C\cup i},v_{C\cup i})=\phi_{\mathbf{c}} (N_{C},v_{C})=:\phi^g ({C};N,v)$ because the Shapley value verifies the null player out property (Derks and Haller, 1999).

With respect to property $P6$, $G$-CBC, let $C\subset N\in {\cal N}$ be any two finite sets, and let $v$ any game in ${\cal G}_N$. First let us remark that condition \eqref{CBC-larga} is equivalent to:
\begin{multline}
\xi^g({C\cup i};N,v)-\xi^g({C\cup j};N,v)=
\\
(\xi^g({C\cup i};N\backslash j,v_{-j})-\xi^g(C;N\backslash j,v_{-j}))-(\xi^g({C\cup j};N\backslash i,v_{-i})-\xi^g(C;N\backslash i,v_{-i})).
\label{CBC-corta}
\end{multline}
Also note that, by definition of the merging game and the Shapley value, for any $i,j\in N\setminus C$ the following equalities hold:
\begin{multline*}
\ds \phi^g({C\cup i};N,v)  =  \ds\sum_{\substack{S\subset N\setminus C \\ i,j\notin S}} \Bigl ( \frac{s!(n-c-s-1)!}{(n-c)!}(v(S\cup C\cup i)-v(S))+
\\ \frac{(s+1)!(n-c-s-2)!}{(n-c)!}(v(S\cup C\cup i\cup j)-v(S\cup j))\Bigr ),
\end{multline*}
\begin{multline*}
\phi^g(C;N\backslash j,v_{-j})=   \ds\sum_{\substack{S\subset N\setminus C\\ i,j\notin S}} \Bigl (\frac{s!(n-c-s-1)!}{(n-c)!}(v(S\cup C)-v(S))+
 \\ \frac{(s+1)!(n-c-s-2)!}{(n-c)!}(v(S\cup C\cup i)-v(S\cup i))\Bigr ),
\end{multline*}
$$
\phi^g ({C\cup j};N\backslash i,v_{-i})= \sum_{\substack{S\subset N\setminus C\\ i,j\notin S}} \frac{s!(n-c-s-2)!}{(n-c-1)!}(v(S\cup C\cup j)-v(S)).
$$
 Analogous expressions hold for $\phi^g({C\cup j};N,v)$, $\phi^g (C;N\setminus i,v_{-i})$ and $\phi^g ({C\cup i};N\setminus j,v_{-j})$. Now it is enough to check that for every $S\in N$ with $i,j\notin S$ the coefficients of $v(S\cup C\cup i\cup j)$, $v(S\cup C\cup i)$, $v(S\cup C\cup j)$, $v(S\cup C)$, $v(S\cup i)$, $v(S\cup j)$ and $v(S)$ are the same in both sides of the equation in \eqref{CBC-corta}, and this is easily deduced from the previous expressions. We leave the details to the reader.

It remains property $P7$, $G$-SPB. Consider the unanimity game with respect to the grand coalition $(N,u_N)$, and a non-empty group $C\subset N\in {\cal N}$. It is straightforward to see that the merging game $(N_C,(u_N)_C)$ is the unanimity game $(N_C,u_{N_C})$ with respect to the grand coalition $N_C$, so $\phi^g (C;N,u_N):=\phi_{\mathbf{c}}(N_C,(u_N)_C)=\frac{1}{n-c+1}$, as desired.
\end{proof}

Next, we give our characterization of the Shapley group value. It must be remarked that it is not a trivial extension of a characterization of the Shapley value, since we do not impose any condition about the value of the individual agents out of the group $C$ we are evaluating. In particular, we have been forced to use in the same characterization group linearity and group coalitional balanced contributions properties. Thus, we should carefully check that all the considered properties are necessary to guarantee the uniqueness of the group value $\phi^g$ (see the Appendix).

\begin{theorem}
The unique group value over the set of all games ${\cal G}$ verifying $G$-null player, $G$-linearity, $G$-CBC, and $G$-SPB ($P3$, $P5$, $P6$ and $P7$, respectively) is the Shapley group value $\phi^g$.
\label{caracterizacion}
\end{theorem}

\begin{proof}
We have proved that the properties hold for the Shapley group value, so  we are left with the question of uniqueness.

Since $\{ (N,u_S)\}_{\substack{S\subset N \\ S\neq\emptyset}}$ forms a basis of ${\cal G}_N$ for all $N\in {\cal N}$, by $G$-linearity it is sufficient to consider the games $(N,u_S)$, $\emptyset\neq S\subset N\in {\cal N}$. So let us see that $\xi^g  (C;N,u_S)=\phi^g (C;N,u_S):=\phi_{\mathbf{c}} (N_C,(u_{S})_C)$ for all non-empty subsets $C,S\subset N\in {\cal N}$. When $C=\emptyset$, this equality trivially holds by definition of a group value.

The proof will consist in a double induction over the cardinality of the player set $N$ (first induction) and the cardinality of the unanimous coalition $S\subset N$ (second induction).

First, we will prove that $\xi^g  (C;N,u_S)=\phi^g (C;N,u_S)$ for all non-empty subsets $C,S\subset N$ whenever the cardinality of $N\in {\cal N}$  is $n\leq 2$.  For the unanimity game $(\{i\},u_i)$ with just one player $i\equiv N$, $G$-SPB property $P7$ implies that $\xi^g (i;N,u_i)=1=\phi^g(i;N,u_i)$.
  Now, let $(N,u_S)$ be a two-person unanimity game with $N=\{i,j\}$. For the unanimity game $u_N$, $P7$ implies that
$\xi^g ( \{i,j\};N,u_N)=\phi^g (\{i,j\};N,u_N)$. For the unanimity game $(N,u_i)$, $G$-null player $P3$ implies:
\begin{align}
\xi^g (j;N,u_i)=\xi^g (\emptyset;N,u_i):=0 \;\text{ and } \; \xi^g (i;N,u_i)=\xi^g (\{i,j\};N,u_i) ,
\label{ec1}
\end{align}
Now, $G$-CBC property $P6$ implies
\begin{align*}
\xi^g (i;N,u_i)-\xi^g (j;N,u_i) & =
(\xi^g (i;\{i \},u_i)-\xi^g({\emptyset};\{i \},u_i)-(\xi^g(j;\{j \}, u^0)-\xi^g({\emptyset};\{j \}, u^0)),
\\
& = 1-0-0+0,
\end{align*}
where $\xi^g(i;N,u_i)=0$ for the trivial game $(\{i\},u^0)$ with $u^0(i)=0$, follows from $P3$. Therefore, taking into account \eqref{ec1}, $\xi^g (\cdot;\{ i,j\}, u_i)\equiv \phi^g (\cdot;\{ i,j\}, u_i)$ holds. The same reasoning applies to $(N,u_j)$.

Let us fix a player set $N\in {\cal N}$ with $\vert N\vert =r$, and consider the unanimity game $(N,u_S)$ for a fix set $\emptyset\neq S\subset N$. We will prove that $\xi^g (C;N,u_S)=\phi^g (C;N,u_S)$ for all $C\subset N$. Two cases are possible:
\begin{enumerate}[$(i)$]
\item
If $S=N$, then $G$-SPB implies $\xi^g (C;N,u_N)=\phi^g (C;N,u_N)$ for any non-empty group $C$ in $N$.
\item
Otherwise, if $\emptyset\neq S\subsetneq N$, we proceed by induction on the cardinality of $C$. Let us first prove the individual case $C=\{i\}$.

There is at least a player $j\in N\backslash S$ which by definition of unanimity game must be null. Let $i$ be a player in $S$. Again by $G$-CBC, taking $C=\emptyset$, we obtain
$\xi^g (i;N,u_S)= \xi^g (i;N\backslash j,{u_S} \vert_{N\setminus j})$,
%
%
since $G$-null player implies $\xi^g (j;N,u_S)=\xi^g ({\emptyset};N,u_S)=0$, and taking into account that $u_S\vert_{N\setminus i}\equiv u^0\in {\cal G}_{N\setminus i}$.

Now, we may assume by the first induction that for every unanimity game $(N',u_S\vert_{N'})$ with $N'\subsetneq N$ we have $\xi^g (C;N',u_S\vert_{N'})=\phi^g (C;N',u_S\vert_{N'})$ for any group $C$ in $N'$. Thus, $\xi^g (i;N\backslash j,{u_S} \vert_{N\setminus j})=
\phi^g (i;N\backslash j,{u_S} \vert_{N\setminus j})$, which in turn is equal to $\frac{1}{s}$ by definition, and then  $\xi^g (i;N,u_S) =\frac{1}{s}=\phi_i(N,u_S)=\phi^g(i;N,u_S)$. Note that every $i\notin S$ is a null player in $u_S$ and therefore $G$-null player implies
$\xi^g (i;N,u_S)=0=\phi^g (i;N,u_S)$. So we are done with the individual case $C=\{i\}$.

Now, in order to prove $\xi^g (C;N,u_S)=\phi^g (C;N,u_S)$ for all groups $C$ with $c>1$, we proceed by induction on the cardinality of $C$ (second induction). So we take now $1<r'\leq r$, and  we may assume that $\xi^g (C;N,u_S)=\phi^g (C;N,u_S)$ holds for any $C\subset N$ with $|C|<r'$. We will check that $\xi^g (D;N,u_S)=\phi^g (D;N,u_S)$ for all $D\subset N$ with $|D|=r'$.

Let $D$ be a fixed subset of $N$ of cardinality $r'$. Since  $S\varsubsetneq N$, again there is a null player $j$ in $u_S$. So, if $j\in D$, then $D\setminus j$ is a coalition of cardinal $r'-1$ and, thus, $G$-null player and the second induction hypothesis imply
 $$
 \xi^g (D;N,u_S)=\xi^g ({D\backslash j};N,u_S)=\phi^g ({D\backslash j};N,u_S)=\phi^g(D;N,u_S).
 $$
 Otherwise, if $D$ does not contain any null player, then let $i$ be a player in $D\subset S$. By the second induction hypothesis and $G$-null player property it holds
 $$
 \xi^g ({(D\backslash i)\cup j};N,u_S)=\xi^g ({(D\backslash i)};N,u_S)=\phi^g ({(D\backslash i)};N,u_S)=\phi^g ({(D\backslash i)\cup  j};N,u_S).
 $$
Note also that $\xi^g (C;N\backslash i,{u_S}\vert_{N\setminus i}))=0=\phi^g (C;N\backslash i,{u_S}\vert_{N\setminus i}))$ for all $C\subset N$, since $u_S\vert_{N\setminus i}\equiv u^0\in {\cal G}_{N\setminus i}$. Hence, by $G$-CBC, taking $C=D\backslash i$, and the first induction,
\begin{multline*}
\xi^g (D;N,u_S)- \phi^g ({(D\backslash i)\cup j};N,u_S) = (\phi^g (D;N\backslash j,{u_S}\vert_{N\setminus j})-\phi^g ({D\backslash i};N\backslash j,{u_S}\vert_{N\setminus j})) -
\\
(\phi^g ({(D\backslash i)\cup j};N\backslash i,{u_S}\vert_{N\setminus i})-\phi^g ({D\backslash i};N\backslash i,{u_S}\vert_{N\setminus i})) =
\phi^g (D;N,u_S)- \phi^g ({(D\backslash i)\cup j};N,u_S).
\end{multline*}
\end{enumerate}
So we have proved the uniqueness for the unanimity games $(N,u_S)$ for all $S\subset N\in {\cal N}$ and we are done.
\end{proof}

\begin{remark}
As we commented before, $G$-SPB leads players to act as one representative in pure bargaining games when all players are strictly necessary, and this combined with $G$-null player and $G$-CBC, lead to the same strategic behavior in all pure bargaining games $(N,u_S)$, $S\subset N$.
In fact, assume we replace the certainty equivalent in $G$-SPB as follows:
\begin{equation}
\xi^g (C;N,u_N)=\frac{c}{n}, \text{ for every non-empty group $C\subset N\in {\cal N}$} .
\label{prop-shapley-suma}
\end{equation}
Then,  condition \eqref{prop-shapley-suma}, $G$-null player, $G$-linearity and $G$-CBC characterize the additive group value ${\cal A}\phi^g$ defined as the sum of the individual Shapley values of the involved players, where ${\cal A}\phi^g(C;N,v):=\sum_{i\in C} \phi_i (N,v)$, for every $C\subset N$ and for all $N\in {\cal N}$.
\end{remark}

We conclude from the previous remark that in case we want to use a group valuation with the Shapley value standards that accounts for the synergy among group members, then we must use the Shapley group value we define.

\begin{remark}
Note also that if we strengthen $G$-SPB imposing the certainty equivalent of playing every pure bargaining game $(N,u_S)$, for all $S\subset N\in {\cal N}$, as follows $(C,u_S) \sim (\mathbf{c}, \frac{1}{\vert S\setminus C\vert +1}u_{\mathbf{c}})$, for every group $C$ such that $S\cap C\neq \emptyset$, then the Shapley group value can be characterized by neutrality to both kind of risks: ordinary and strategic.
\label{sobre-axiomas}
\end{remark}

\section{Profitability of a group}
\label{profitability}

In Section \ref{def-axiomatizacion} we have proposed a method to measure the a priori value of a group and we have characterized it axiomatically. In this section we rely on this valuation to analyze the {\em profitability} of a group.

Next proposition shows that supperadditivity implies that the expected value of group $C$ is at least the value that their members can assure for themselves. Moreover, if the game is monotonic larger groups are more valuable.

We also analyze when the integration of group $C$ is {\em mergeable} in the sense of Derks and Tijs (2000). We end up this section studying the marginal effect that the incorporation of a new member has over an already integrated group. Theorem \ref{partner-contribution} relates this marginal effect with a measure of average complementarity between the entrant player and the incumbent group.

\begin{proposition}
Let $N\in {\cal N}$ be any finite set of players, and $v$ be any game in ${\cal G}_N$. Then, the Shapley group value $\phi^g$ verifies the following properties:
\begin{enumerate}[$(i)$]
\item
{\em Group Rationality}: $\phi^g (C;N,v) \geq v(C)$ for every $C\subset N \in {\cal N}$ if the game $v\in {\cal G}_N$ is superadditive, and
\item
{\em Monotonicity}: $\phi^g  (C;N,v) \leq  \phi^g (D;N,v)$ for every pair of coalitions $C\subset D\subset N\in {\cal N}$ if the game $v\in {\cal G}_N$ is monotonic.
\end{enumerate}
\end{proposition}

\begin{proof}
Group rationality follows from the individual rationality of the Shapley value. Note that every merging game $(N_C,v_C)$, $C\subset N$, is superadditive if it is $(N,v)$, for all $C\subset N\in {\cal N}$ and $v\in {\cal G}_N$.

Monotonicity follows from being
\begin{multline}
\phi^g ({C\cup i};N,v) - \phi^g (C;N,v)=\sum_{\substack{S\subset N\setminus C\\ i\notin S}} \frac{s! (n-c-s)!}{(n-c+1)!} \Bigl ( v(S\cup i\cup C)-v(S\cup C)\Bigr ) +
\\
 \frac{(s+1)! (n-c-s-1)!}{(n-c+1)!} \Bigl ( v(S\cup i)-v(S) \Bigr ) \geq 0,
\label{marginal-contribution-iC}
\end{multline}
for all coalitions $C\subset N\in {\cal N}$, and all players $i\notin C$, whenever the  game $v$ is monotonic.
 \end{proof}

Now we will examine the profitability of the integration of group $C$ measured as the difference between the Shapley value of group $C$, which represents its a priori valuation when they act as one representative, and the sum of the individual Shapley values of the involved players (the additive Shapley group value of $C$), i.e.
$$
\phi^g(C;N,v)-\sum_{i\in C} \phi_i(N,v) .
$$
Profitability is analyzed by Derks and Tijs (2000), and also by Segal (2003). Next, we recall the general results they obtained.

\begin{proposition}[Derks and Tijs, 2000]
Let $N\in {\cal N}$ be any finite set of players, and $v$ be any game in ${\cal G}_N$. Then, coalition $C\subset N \in {\cal N}$ is {\em profitable}\footnote{Actually, Derks and Tijs (2000) refer to profitability as {\em mergeability}.}
 whenever all coalitions with positive Harsanyi dividend are either contained in $C$ or have at most one player in common with $C$.
\label{DT-mergeability-theorem}
\end{proposition}

Derks and Tijs also propose some interesting types of games for which every coalition is profitable, or profitability  can be guaranteed for certain kinds of coalitions.

The results of Segal (2003) rely on the  {\em second-order difference operator} for a pair of players $i,j\in N$, and the {\em third-order difference operator} for three players $i,j,k\in N$.

The {\em second-order difference operator} for a pair of players $i,j\in N$ is defined as a composition of marginal contribution operators (i.e., first-order difference operators) as follows:
$$
\Delta^2_{ij} (S;N,v)=v(S\cup\{i,j\})-v(S\cup  j ) - v(S\cup  i)+v(S),\quad\forall\; S\subset N\setminus\{ i,j\} .
$$
Here  $\Delta^2_{ij} (S;N,v)$ expresses player $i$'s effect over the marginal contribution of player $j$ (or vice versa). Note that
 $v(S\cup\{i,j\})-v(S)=\Delta^2_{ij} (S;N,v)+\Delta_{i} (S;N,v)+\Delta_{j} (S;N,v)$, and thus $\Delta^2_{ij} (S;N,v) > 0$ implies that the marginal contribution of players $i,j$ as a group exceeds the sum of the individual marginal contributions of each player.

In fact, following Bulow, Geanakoplos, and Klemperer (1985),  players $i$ and $j$ are said to be {\em strategic complements} whenever  $\Delta^2_{ij} (S;N,v)\geq 0$, for all $S\subset N\setminus \{ i,j\}$.  They are said to be {\em strategic substitutes} whenever $\Delta^2_{ij} (S;N,v)\leq 0$, for all $S\subset N\setminus \{ i,j\}$. Therefore, $\Delta^2_{ij} (S;N,v)$  can be interpreted as a measure of players $i$ and $j$ {\em interaction} with respect to the players in $S$.

Analogously, the {\em third-order difference operator} for players $i,j,k\in N$ is defined as $\Delta^3_{ijk} (\cdot;N,v)=\Delta_i(\Delta^2_{jk} (\cdot;N,v))$, for all $S\subset N\setminus\{ i,j,k\}$.
 Here  $\Delta^3_{ijk} (S;N,v)$ expresses player $k$'s effect over the complementarity between players $i$ and $j$ with respect to the players in $S$. Again, the operator does not depend on the order of taking differences.

Segal (2003) obtains the following result about profitability of groups of two players\footnote{Note that the kind of group integration we work with is equivalent to the {\em collusive} contracts considered by Segal (2003).}, showing that the merging of two players $i,j\in N$ is profitable (unprofitable) whenever the presence of the outside players reduces (increases) the complementarity between the colluding players.
 This author also takes into account the profitability of the incorporation of a new member $j\in N\setminus C$ to an already integrated group $C$ (point $(ii)$ in next Proposition \ref{Segal-collusion-prop}). In this case, profitability is measured with respect to the situation in which the players of group $C$ are colluding. That is, profitability means
\begin{equation}
\phi^g(C\cup j;N,v) \geq \phi^g(C;N,v) + \phi_j(N_C,v_C) .
\label{marginal-contribution-quotient}
\end{equation}

\begin{proposition}[Segal, 2003] Let $N\in {\cal N}$ be any finite set of players, and $v$ be any game in ${\cal G}_N$. Then:
\begin{enumerate}[$(i)$]
\item
 A coalition $C=\{i,j\}\subset N$ of two players is {\em profitable} ({\em unprofitable}) if $\Delta^3_{ijk} (S;N,v) \leq (\geq ) \; 0$, for every coalition $S\subset N\setminus \{ i,j,k\}$, and for all $k\in N\setminus C$. If the reverse inequalities hold, then group $C$ is unprofitable.
\item
The union of the integrated group $C\subset N$ and player $j\notin C$ is profitable (unprofitable) if $\Delta^3_{ijk} (S;N,v) \leq (\geq ) \;0$, for every coalition $S\subset N\setminus \{ i,j,k\}$, and for all $i\in C$, $k\in N\setminus (C\cup i)$.
\end{enumerate}
 \label{Segal-collusion-prop}
 \end{proposition}

Condition $(i)$ above states that profitability of the merging of players $i$ and $j$ is not directly related with their own complementarity. In fact, if we analyze the {\em games with indispensable players} application of Segal (2003) that models for instance the case of a firm that is indispensable to its workers, it holds that the union of two substitutable workers is profitable if their bargaining opponents (the remaining workers) enhances their substitutability. With respect to the union of the firm $p$ and one of the workers $d$, which are always strategic complements, this union is profitable (unprofitable) if worker $d$ is a strategic substitute (complement) of the rest of workers.

According to Segal's results it is clear that complementarity and substitutability are not directly related to profitability. However,
  Theorem \ref{partner-contribution} below shows how the Shapley group value incorporates those kind of relations among the players when evaluating the value of a group. The marginal contribution of a new entrant $j$ to the incumbent group $C$ is the sum of the a priori value of player $j$ which does not depend on $C$, and the {\em average complementarity} between $j$ and $C$. Formally:

  \begin{definition}
Let $N\in {\cal N}$ be any finite set of players, and $v$ be any game in ${\cal G}_N$, the {\em average complementarity} of players $i,j\in N$ is defined as the following average of second-order differences:
\begin{equation}
\psi_{ij} (N,v):=\sum_{S\subset N\setminus \{i,j\}} \frac{s! (n-s-1)!}{n!}  \Delta^2_{ij} (S;N,v), \quad\text{for all $i\neq j\in N$.}
\end{equation}
\label{average-com}
\end{definition}

The average $\psi_{ij} (N,v)$ is taken over all the possible orders of $N=\{1,\dots,n\}$, when all orders are considered equally probable. The second-order difference $\Delta^2_{ij}(S;N,v)$ is considered in all orders in which coalition $S$ contains all players arriving between $i$ and $j$, and $i$ comes before $j$. It can be interpreted as an {\em interaction index} in the sense of Grabisch and Roubens (1999).

\begin{theorem}
Let $N\in {\cal N}$ be any finite set of players, and $v$ be any game in ${\cal G}_N$. Let $C\subset N$ be any group in $N$, and let $i\notin C$. Then, the {\em marginal contribution} of player $i\in N\setminus C$ to the Shapley group value of $C$ equals:
\begin{equation}
MC_i^{g} (C;N,v):=\phi^g ({C\cup i};N,v)-\phi^g (C;N,v)=\phi_i (N\setminus C, v\vert_{N\setminus C}) + \psi_{\mathbf{c}i} (N_C,v_C).
\label{partner-contribution-i}
\end{equation}
\label{partner-contribution}
\end{theorem}

\begin{proof}
Let us arbitrarily fix the sets $C$ and $N$ and the player $i$, $C\subsetneq N\in {\cal N}$, $i\in N\setminus C$, and let $v\in {\cal G}_N$. Then, adding and subtracting the amount
$$
\sum_{\substack{S\subset N\setminus C\\ i\notin S}} \frac{s! (n-c-s-1)!}{(n-c)!} \Bigl ( v(S\cup i)-v(S)\Bigr )
$$
to the above expression \eqref{marginal-contribution-iC} of the marginal contribution $MC_i^{g} (C;N,v)$ follows that:
\begin{align*}
MC_i^{g} (C;N,v)= & \sum_{\substack{S\subset N\setminus C\\ i\notin S}} \frac{s! (n-c-s-1)!}{(n-c)!} \Bigl ( v(S\cup i)-v(S)\Bigr ) +
\\
& \sum_{\substack{S\subset N\setminus C\\ i\notin S}} \frac{s! (n-c-s)!}{(n-c+1)!} \Bigl ( v(S\cup i\cup C)-v(S\cup C)-v(S\cup i) + v(S)\Bigr ).
\end{align*}
The first term is precisely the Shapley value of player $i$ in the restricted game $(N\setminus C, v\vert_{N\setminus C})$, where players in $C$ do not play a role. The second term can be expressed by means of the second-order difference operators for the pair of players $i,\mathbf{c}\in N_C$, as follows:
\begin{multline*}
 \sum_{\substack{S\subset N\setminus C\\ i\notin S}} \frac{s! (n-c-s)!}{(n-c+1)!} \Bigl ( v(S\cup i\cup C)-v(S\cup C)-v(S\cup i) + v(S)\Bigr )=
 \\
 \sum_{\substack{S\subset N\setminus C\\ i\notin S}} \frac{s! (n-c-s)!}{(n-c+1)!}  \Delta^2_{ic} (S;N_C,v_C)=:\psi_{\mathbf{c}i} (N_C,v_C) .
\end{multline*}
\end{proof}

The previous result shows that the value of a group results from a complex combination of independence and complementarity among its members.
In particular, it is clear that  the most valuable $k$ agents from an individual point of view do not form in general the most valuable group of $k$ agents. Let us illustrate this fact with the following simple example.

\begin{example}
Let us consider the following social network represented in Figure \ref{estrellas} as an {\em undirected graph} $(N,\Gamma)$, and the {\em connectivity game} (Amer and Gimenez, 2004), which is defined as
$$
v(S)=\begin{cases} 1, & \text{if $S$ is connected in $\Gamma$ and $\vert S\vert >1$,}
\\
0, & \text{otherwise,}
\end{cases}, \quad \text{for all $S\subset N$.}
$$

\begin{figure}[h]


\begin{center}
\epsfig{file=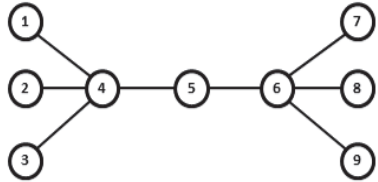,width=5cm,height=3.75cm}
\end{center}

\caption{Social Network $(N,\Gamma)$}
\label{estrellas}
\end{figure}

Here, $S\subset N$ is a {\em connected} coalition in $(N,\Gamma)$ if for every two players $i\neq j$ in $S$, $\{i,j\}\in \Gamma$, or there exists a {\em path} between them which consists of nodes in $S$. That is, there exists a sequence of nodes and edges $\pi(i,j)=\{i=i_1,i_2,\dots, i_{k-1}, i_k=j\}$, with $k\geq 2$ satisfying the property that for all $1\leq r\leq k-1$, $\{i_r,i_{r+1}\}\in \Gamma$, and $i_r\in S$, for all $2\leq r\leq k-1$.

In that case, the two more valuable players, according to their individual Shapley values are the two centers of the satellite stars, players 4 and 6.
$\phi_i(N,v)=-\frac{8}{360}$, for all the leaves $i=1,2,3,7,8,9$, $\phi_4(N,v)=\phi_6(N,v)=\frac{139}{360}$, for the two centers, and $\phi_5(N,v)=\frac{130}{360}$ for the hub which intermediates between players 4 and 6. However, the most valuable group of two agents is the one composed by the hub and one out of the two centers. In fact,
\begin{align*}
\phi^g({\{4,6\}};N,v)  & =\phi_4(N,v) + \phi_6 (N\setminus 4,v\vert (N\setminus 4)) + \psi_{46}(N,v)= \frac{139}{360} + \frac{1}{8}-\frac{1}{90} =
 \frac{1}{2},
\\[3mm]
\phi^g({\{4,5\}};N,v)  & =\phi_4(N,v) + \phi_5 (N\setminus 4,v\vert (N\setminus 4)) + \psi_{45}(N,v)=  \frac{139}{360} + \frac{1}{56}+\frac{19}{72} =
 \frac{1}{2}+\frac{47}{280} .
\end{align*}
The result of this computation can be interpreted in the following lines: although the power of player 5 depends more on the presence of player 4 than the power of 6, the average complementarity of players 4 and 5 is greater than the average complementarity of players 4 and 6.
\blok
\end{example}

\section{Application: Detecting a target group in terrorist networks}
\label{ejemplos}

We will illustrate now the application of the Shapley group value to the two terrorist networks which have been considered by Lindelauf {\em et al.} (2013): the operational network of Jemaah Islamiyah's Bali bombing and the network of hijackers of Al Qaeda's 9/11 attack.

For the first case, Jemaah Islamiyah's Bali bombing attack, the authors use the game $(N,v^{wconn})$. Let $(N,\Gamma)$ be the {\em undirected graph} which represents the terrorist network. The nodes in the  finite set $N=\{1,\dots,n\}$ are  the terrorists, whereas the {\em edges} -i.e., unordered pairs of distinct nodes- represent the known relationships between the terrorists. In Figure \ref{bali-red} the terrorist network we work with is represented.

\begin{figure}[h]


\begin{center}
\epsfig{file=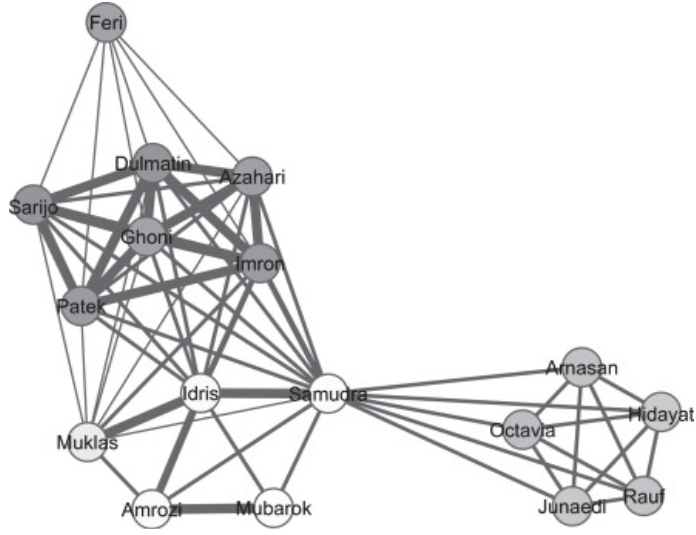,width=10cm,height=7cm}
\end{center}

\caption{Operational network of JI's Bali attack. Image taken from Lindelauf {\it et al.} (2013)}
\label{bali-red}
\end{figure}

Then, Lindelauf {\it et al.} (2013) define the game $(N,v^{wconn})$, which extends the connectivity game of Amer and Gimenez (2004) using information about relationships. In that game, a coalition must be connected in order to achieve a non-zero value. That is, players in coalition $S$ must rely only upon their own connections in order to communicate among themselves. Then, since a terrorist cell tries to prevent discovery during the planning and execution phase of an attack, and taking into account the available data about the existing relationships\footnote{The authors collected the strength of existing relationships from Koschade (2006).}, the authors define the power of a coalition as the total number of relationships that exist within
that coalition divided by the sum of the weights (representing frequency and duration of interaction) on those relationships;
\begin{equation}
 v^{wconn} (S)  = \begin{cases}\frac{\sum_{\substack{i,j\in S \\ i\neq j}} I_{ij}}{\sum_{\substack{i,j\in S \\ i\neq j}} f_{ij}}, &
 \text{if $S$ is connected in $\Gamma$ and $\vert S\vert >1$,}
\\
0, & \text{otherwise,}
\end{cases}, \quad \text{for all $S\subset N$,}
 \label{lindelauf-Bali}
\end{equation}
where $f_{ij}$ is the weight assigned to relation $\{i,j\}\in \Gamma$  in the terrorist network, $I_{ij}=1$, for every edge $\{i,j\}$ in $\Gamma$, and 0 otherwise.

We obtain the following results concerning groups from one to four individuals. Following Castro, Gomez and Tejada (2009), and taking into account that the marginal contributions in the extended connectivity games are computable in polynomial time, we have estimated with Monte Carlo simulation
 the Shapley group value of the examples, also in polynomial time. The results obtained are represented in Table \ref{resultados-bali}, which includes the records for the best groups arranged in decreasing order of importance.

\begin{table}[h]

{\scriptsize
\begin{center}
\begin{tabular}{|c|c|c|c|}
\hline
Individuals & Two agents & Three agents & Four agents
\\
\hline
Samudra, $0,358$ & $\{ \text{Samudra, Muklas}\}$, $0,453$ &
  $\{\text{Samudra, Muklas,Azahari}\}$, $0,442$ &
$\{\text{Samudra, Muklas,Feri,Azahari}\}$, $0,466$
\\
Muklas, $0,048$ & $\{ \text{Samudra, Azahari}\}$,  $0,392$ &
$\{\text{Samudra, Muklas,Sarijo}\}$, $0,435$ &
$\{\text{Samudra, Muklas,Feri,Sarijo}\}$, $0,460$
\\
Feri, $0,032$ & $\{ \text{Samudra, Sarijo}\}$,  $0,386$ &
$\{\text{Samudra, Muklas,Patek}\}$, $0,435$ &
$\{\text{Samudra, Muklas,Feri,Patek}\}$, $0,460$
\\
Azahari, $0,012$ & $\{ \text{Samudra, Patek}\}$,  $0,386$ &
$\{\text{Samudra, Feri,Azahari}\}$, $0,430$ &
$\{\text{Samudra, Muklas,Feri,Ghoni}\}$, $0,453$
\\
Sarijo, $0,005$ & $\{ \text{Samudra, Rauf}\}$,  $0,384$ &
$\{\text{Samudra, Muklas,Ghoni}\}$, $0,429$ &
$\{\text{Samudra, Muklas,Azahari,Sarijo}\}$, $0,429$
\\
\hline
\end{tabular}
\end{center}
}

\caption{Operational network of JI's Bali attack rankings}
\label{resultados-bali}
\end{table}

According to the individual rankings for the JI network based on the Shapley value, the five most valuable terrorists were, in decreasing order of importance: Samudra, Muklas, Feri, Azahari and Sarijo.

With respect of groups of two terrorists, the most valuable group is that composed by the two more important agents, $\{ \text{Samudra, Muklas}\}$. However, the second group of size two in importance is $\{ \text{Samudra, Azahari}\}$, improving the group value of $\{ \text{Samudra, Feri}\}$, which equals $0,350$, and takes the 15th place. In fact, Samudra has all direct contacts has Feri, and therefore Feri's presence in a group is somehow redundant if Samudra is already in it. According to what it is known about the attack, ''{\em Samudra, an engineering graduate, played a key role in the bombings}'', whereas Azahari is the bomb expert who was considered the ''brain'' behind the entire operation.

Again, the most valuable group of three terrorist is $\{\text{Samudra, Muklas,Azahari}\}$. However, when considering a bigger group of four terrorist, then $\{\text{Samudra, Muklas, Feri, Azahari}\}$ has the highest Shapley group value.

In the analysis of the terrorist network of the 11S, Lindelauf {\it et al.}'s  starting point was the version of the network in Figure \ref{redgrande}, whose links come from terrorists that lived or learned together (black edges) as well as some temporary links that were only activated just before the attack in order to coordinate the cells. See Krebs (2002) for further information. The authors use the game $(N,v^{wconn2})$, which uses information about the individuals:
 \begin{equation}
 v^{wconn2} (S)   =  \begin{cases} \displaystyle \sum_{i\in S} w_{i}, &
 \text{if $S$ is connected in $\Gamma$ and $\vert S\vert >1$,}
\\
0, & \text{otherwise,}
\end{cases}, \quad \text{for all $S\subset N$,}
 \label{lindelauf-w2}
\end{equation}
where $w_i$ is the weight assigned to terrorist $i\in N$. The authors also determine the terrorist weights in their analysis (see Table 5 in Lindelauf {\it et al.}, 2013).

\begin{figure}[h]


\begin{center}
\epsfig{file=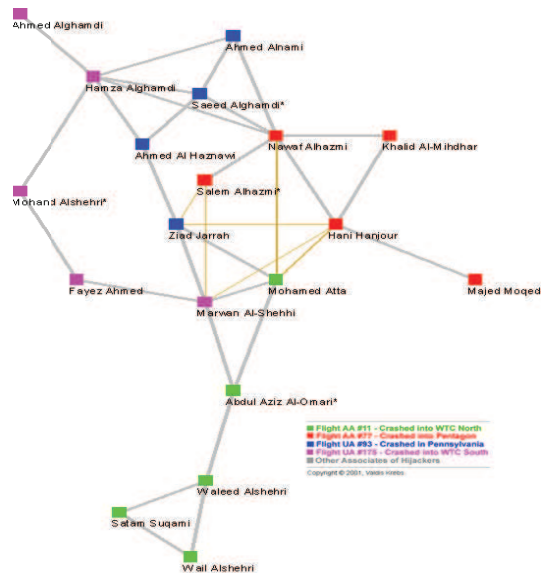,width=7cm,height=8.25cm}
\end{center}

\caption{11S Social Network. Image taken from V.E. Krebs (Copyright $\copyright$2002, First Monday)}
\label{redgrande}
\end{figure}

 The results obtained  by means of Monte carlo simulation (see Castro, Gomez and Tejada, 2009) are depicted in
 Table \ref{resultados-11S}, which includes the records for the best groups arranged in decreasing order of importance.

\begin{table}[h]

{\scriptsize
\begin{center}
\begin{tabular}{|c|c|c|}
\hline
Individuals & Two agents & Three agents
\\
\hline
A. Aziz Al-Omari (WTC-N), $6,096$ & \{Al-Omari, Al-Ghamdi\}, $7,405$ & \{Al-Omari, H. Hanjour, M. Al-Sehhi\}, $9,236$
\\
H. Al-Ghamdi (WTC-S), $5,578$ & \{Al-Omari, M. Al-Sehhi\}, $7,392$ & \{Al-Omari, H. Al-Ghamdi, Wd. Al-Shehri\}, $9,153$
\\
Wd. Al-Shehri (WTC-N), $5,563$ & \{Al-Omari, H. Hanjour\}, $7,368$ & \{M. Al-Shehhi, H. Al-Ghamdi, Wd. Al-Shehri\}, $9,140$
\\
H. Hanjour (Pent), $5,402$ & \{ Aziz Al-Omari, Wd. Al-Shehri\}, $7,324$ & \{ H. Al-Ghamdi, Wd. Al-Shehri, H. Hanjour\}, $9,074$
\\
M. Al-Shehhi (WTC-S), $2,202$ & \{Al-Omari, M. Atta\}, $7,156$ & \{ H. Al-Ghamdi,  H. Hanjour, N. Al-Hazmi\}, $8,986$
\\
M. Atta (WTC-North), $1,600$ & \{ H. Al-Ghamdi, H. Hanjour\}, $7,044$ &  \{ Al-Omari, H. Al-Ghamdi, H. Hanjour\}, $8,963$
\\
\hline
\end{tabular}
\end{center}
}

\caption{11S-hijackers network rankings}
\label{resultados-11S}
\end{table}

According to the individual rankings for the Al Qaeda's 9/11 network based on the Shapley value, the most valuable terrorists were, in decreasing order of importance: A. Aziz Al-Omari (WTC North cell), H. Al-Ghamdi (WTC South cell), Wd. Al-Shehri (WTC North cell), H. Hanjour (Pentagon cell), M. Al-Shehhi (WTC South cell) and M. Atta (WTC North cell).

With respect of groups of two terrorists, the most valuable group is that composed by the two more important agents, \{A. Aziz Al-Omari, Al-Ghamdi\}. However, the second group of size two in importance is \{Aziz Al-Omari, M. Al-Sehhi\}, improving the group value of \{Aziz Al-Omari, Wd. Al-Shehri\}, which takes the 4th place. In fact, Wd. Al-Shehri is one out of the three hijackers that crashed the plane into WTC South which forms a cycle in the terrorist network that is connected to the rest of terrorist only via A. Aziz Al-Omari, who also  belongs to the WTC North cell. Thus, Wd. Al-Shehri presence in a group is not so necessary if A. Aziz Al-Omari is already in it.

The most valuable group of three terrorist is \{A. Aziz Al-Omari, H. Hanjour, M. Al-Sehhi\}. In that case, Wd. Al-Shehri, from WTC North cell, and  H. Al-Ghamdi, from WTC South cell, are displaced. Now, H. Hanjour, who is known to be the leader of WTC South cell, has displaced H. Al-Ghamdi.

When considering a bigger group of four terrorist, then \{A. Aziz Al-Omari, H. Hanjour, M. Al-Sehhi, Wd. Al-Shehri\} has the highest Shapley group value. The first group with one representative per each cell occupies the 13th place, with a Shapley group value of $\phi^g ({D};N,w^{conn2})=10,6311$, being \{A. Aziz Al-Omari, H. Hanjour, M. Al-Sehhi, Z. Jarrah\}. Note that Z. Jarrah, who is known to be the leader of Pennsylvania cell is individually in the 8th position. The group of the four cell's leaders $L=\{\text{M. Atta, H. Hanjour, M. Al-Sehhi, Z. Jarrah}\}$ is in the 48th position with a group value of $\phi^g ({L};N,w^{conn2})=9,1313$.

Recall that Lindelauf {\it et al.} (2013) carried out the analysis on the terrorist network of the nineteen hijackers which prepared and executed the attack (distributed in four cells). We extend their analysis to a more dense network (see Figure \ref{redextendida}) which included some people which did not take direct part in the attack, but support the terrorists. In this event, the relative positions of the  hijackers change: the two poor connected terrorists from the WTC North cell, A. Aziz Al-Omari and Wd. Al-Shehri, are not so relevant in the new network, since they are now better connected through non-hijackers terrorists. According to the rankings for the Al Qaeda's 9/11 hijackers based on the individual Shapley value and the extended network\footnote{In which we have again considered the $w^{conn2}$ game, with a  zero weight for all the terrorists who do not take direct part in the attacks.}, the most valuable hijackers were, in decreasing order of importance, N. Al-Hazmi (Pentagon), H. Hanjour (Pentagon), M. Atta (WTC North), H. Al-Ghamdi (WTC South), Wd. Al-Shehri (WTC North) and Z. Jarrah (Pennsylvania).

\begin{figure}[h]


\begin{center}
\epsfig{file=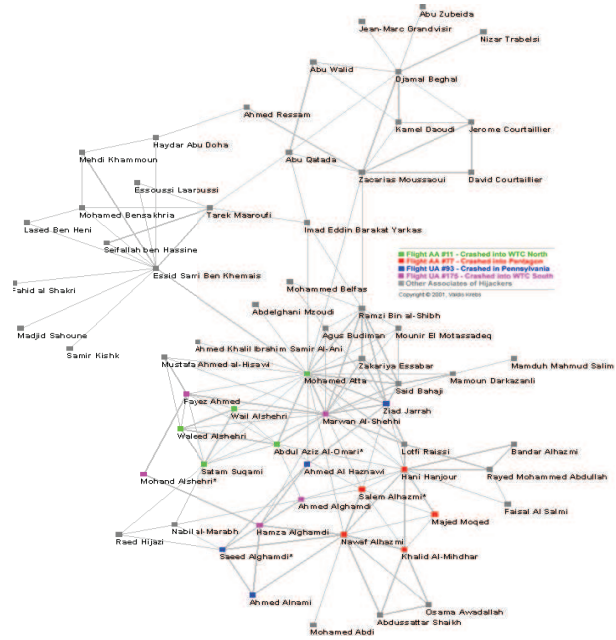,width=8cm,height=10.00cm}
\end{center}

\caption{11S extended SN. Image taken from V.E. Krebs (Copyright $\copyright$2002, First Monday)}
\label{redextendida}
\end{figure}

 The results obtained  by means of Monte carlo simulation (see Castro, Gomez and Tejada, 2009) are depicted in
 Table \ref{resultados-11S-extendida}, which includes the records for the best groups arranged in decreasing order of importance.

\begin{table}[h]

\begin{center}
\begin{tabular}{|c|c|}
\hline
Individuals & Two agents
\\
\hline
 N. Al-Hazmi (Pent), $6,132$ & \{ N. Al-Hazmi, M. Atta\}, $8,424$
 \\
 H. Hanjour (Pent), $6,089$ & \{ N. Al-Hazmi, H. Hanjour\}, $8,342$
 \\
 M. Atta (WTC-N), $5,926$ & \{  H. Hanjour,  M. Atta\}, $8,152$
 \\
H. Al-Ghamdi (WTC-S), $1,844$ & \{ N. Al-Hazmi, H. Al-Ghamdi\}, $7,201$
 \\
 Wd. Al-Shehri (WTC-N), $1,701$ & \{ H. Hanjour, Z. Jarrah\}, $7,048$
 \\
Z. Jarrah (Penn), $1,688$ & \{ M. Atta, Z. Jarrah\}, $7,013$
 \\
\hline
\end{tabular}
\end{center}

\caption{11S extended network rankings}
\label{resultados-11S-extendida}
\end{table}

With respect to groups of two terrorists, the most valuable group is  $\{ \text{N. Al-Hazmi, M. Atta}\}$,  which does not coincide with the two more important terrorists' group. In that case, however, the three and four people most valuable 3-group and 4-group are composed of the three and four, respectively, most important hijackers from an individual point of view. Those groups, $C_3$ and $C_4$, have group values of $\phi^g ({C_3};N\cup M,w^{conn2})=10,675$ and $\phi^g ({C_4};N\cup M,w^{conn2})=12,529$. Now, the group $L$ is in the twentieth position with a value of $\phi^g ({L};N\cup M,w^{conn2})=10,066$.

\section{Conclusions}
\label{conclusions}

In this paper we have introduced an extension of the classical game-theoretic concept of value to the framework of groups. We define the group value as a vector whose components measure the expectations of the coalition inside a game, i.e. an evaluation of the prospects of the group if they act together as an individual; in this sense, the group value must be understood as a {\emph valuation vector}. A key observation in our proposal is that we do not need to suppose necessarily that the agents know each other nor agree to act jointly; instead, we assume the existence of an {\emph external decision maker} that is able and willing to evaluate the value of the different groups, and in his case, find the optimal group who undertakes the appropriate collective action he is trying to promote. For instance, to find the group of a given size $k$ with the largest Shapley group value.

The main motivation of our work and in particular of the previous definitions was to obtain a (marginalistic) extension of the Shapley value to the context of groups. Following the original formulation of Shapley, who intended to apply his value to measure the expectations of players in a game, and also keeping in mind the mentioned idea of the external decision maker, we have performed the generalization of the Shapley value by means of the merging game defined by Derks and Tijs in 2000. In this work, the authors develop a concept of super-player, who acts as a proxy in a merging game of all the players of the coalition whose value we do want to compute.

In order to show that our extension of Shapley is valid and interesting, we have proved that the natural generalizations to the group framework of the usual properties of the (individual) Shapley value hold for the Shapley group value. Moreover, we offer an axiomatic characterization of the Shapley group value which, although includes the group version of null-player and linearity properties, cannot be directly deduced from the usual
characterizations in the individual context. We analyze the implications of the considered axioms. In particular, $G$-SPB leads players to act as one representative in pure bargaining games when all players are strictly necessary, and this combined with $G$-null player and $G$-CBC, lead to the same strategic behavior in all pure bargaining games $(N,u_S)$, $S\subset N\in {\cal N}$. Then, linearity extends this behavior to every game. Moreover, considering an alternative formulation of $G$-SPB we characterize an additive Shapley group valuation. Therefore, we conclude that in case we want to use a group valuation with the Shapley value standards that accounts for the synergy among group members, then we must use the Shapley group value we define. It is likely other axiomatic descriptions of the Shapley group value may exist.

Based on our proposal, and following ideas of Segal, we have elaborated about the ideas of complementarity and substitutability that concerns to profitability of acting as a group. However, it is not easy to find a straight relation between complementarity/substitutability and profitability (or not) because this relation only could be found by means of Segal third derivative, and not using only the second derivative.

We finish our paper by testing the validity of our methods in the identification of influent groups inside a real terrorist network. The flexibility of the proposed approach allows to suppose that our measure will be effective and usual in a variety of contexts and making use of different interpretations of the Shapley group value. Let us think for instance in two of them.

In a global economy context, in which many firms present a complex interlocked shareholding structure, it may be difficult to asses a firm's controllers. However, ``{\em a common intuition among scholars and in media sees the global economy as being dominated by a handful of powerful transnational corporation}'' (Vitali, Glattfelder and Battiston, 2011). In that case, following a game theoretical approach, we can make use of the Shapley group value to detect a small group of firms which in fact have a dominating power. The reader is referred to Crama and Leruth (2013) for an interesting review about this approach.

Another relevant application arises in the context of transportation network's operation, where the identification of sets of stations to defense (or maintain) in order to maximally preserve the network's operation is a relevant question to network protection against natural and human-caused hazards, which has become a topical research topic in engineering and social sciences, as Liu, Fan and Ordo\~nez (2009) point out. In that case, following a game theoretical approach, the Shapley group value can serve security agencies for selecting a group of stations to defense (or maintain).

However, it should be remarked that the problem of finding the optimal group, according to some prearranged criteria, is a combinatorial problem that merits a more careful study. We are aware of the need of heuristics in order to apply the Shapley group value to the group selection problem.

\section*{Appendix}

\vspace*{3mm}
In the following we should check that all the properties considered in Theorem \ref{caracterizacion} are necessary to guarantee the uniqueness of the Shapley group value $\phi^g$.

\vspace*{5mm}
\noindent \textbf{$G$-null player}. Let $\alpha\in (0,1)$. Define a group value $\xi^g$ in the following manner.
Let $N$ be any finite set of players $N\in {\cal N}$, and let $v\in {\cal G}_N$. If $(N,v)$ is a null game, $\xi^g(C;N,v)$=0 for every $C\in N$. Given a non-null unanimity game $(N,u_S)$, with $S\subsetneq N$, define $\xi^g(C;N,u_S)$ as
\begin{equation}
\xi^g(C;N,u_S)=\begin{cases} \ds \sum_{k=0}^{c-1} \alpha^{n-k}, & \text{ if $C\cap S=\emptyset$,}
\\
\ds \frac{1}{s-\vert S\cap C\vert +1} + \sum_{k=0}^{n-s-1} \alpha^{n-k}, & \text{ if $C\cap S\neq \emptyset$.}
\end{cases}
\end{equation}
For $S=N$, take $\xi^g(C;N,u_N)=\frac{1}{n-c+1}$ for all $C\subset N$, and then extend the group value to ${\cal G}_N$ by linearity. It is clear that $\xi^g (i;N,u_S)=\alpha^n >0$ for all $i\notin S$. So, $G$-null player does not hold. Let us check that $\xi^g$ verifies $G$-CBC over the class of unanimity games. If $(N,u_N)$, then $G$-CBC condition \eqref{CBC-corta} trivially holds. Let $(N,u_S)$ be a unanimity game with $S\subsetneq N$. If $\vert S\vert =n-1$, then
 two cases are possible:
 \begin{enumerate}[a)]
 \item
 If $i,j\in S$, then $\xi^g ({C\cup i};N,u_S)= \frac{1}{s-\vert S\cap C\vert} + \sum_{k=0}^{n-s-1} \alpha^{n-k}=\xi^g ({C\cup j};N,u_S)$ and \eqref{CBC-corta} holds, since $u_S\vert_{N\setminus i}$ and $u_S\vert_{N\setminus j}$ are null games.
 \item
 If $i\in S$ and $j\notin S$, then $S=N\setminus j$, $u_S\vert_{N\setminus i}\equiv 0$ and $u_S\vert_{N\setminus j}$ is a unanimity game w.r.t. the grand coalition $N\setminus j$. Thus \eqref{CBC-corta} holds:
 \begin{multline*}
 \xi^g ({C\cup i};N\setminus j, u_{N\setminus j}\vert_{N\setminus j})-\xi^g (C;N\setminus j, u_{N\setminus j}\vert_{N\setminus j}) =
 \\[3mm]
 \frac{1}{(n-1)-(c+1)+1} -\frac{1}{(n-1)-c+1}=\xi^g ({C\cup i};N, u_{N\setminus j})-\xi^g ({C\cup j};N, u_{N\setminus j}) .
 \end{multline*}
 \end{enumerate}
If $\vert S\vert <n-1$, then  three cases are possible:
 \begin{enumerate}[a)]
 \item
 If $i,j\in S$, then $\xi^g ({C\cup i};N,u_S)= \frac{1}{s-\vert S\cap C\vert} + \sum_{k=0}^{n-s-1} \alpha^{n-k}=\xi^g ({C\cup j};N,u_S)$ and \eqref{CBC-corta} holds.
 \item
 If $i\in S$ and $j\notin S$, then
 $$
 \xi^g ({C\cup i};N\setminus j, u_{S}\vert_{N\setminus j}) = \frac{1}{s-\vert S\cap C\vert} + \sum_{k=1}^{n-s-1} \alpha^{n-k}
 $$
 and
 $$
 \xi^g (C;N\setminus j, u_{S}\vert_{N\setminus j}) = \begin{cases}
 \ds \sum_{k=1}^{c} \alpha^{n-k}, & \text{ if $C\subset N\setminus S$,}
 \\
  \ds \frac{1}{s-\vert S\cap C\vert +1} + \sum_{k=1}^{n-s-1} \alpha^{n-k}, & \text{ otherwise.}
 \end{cases}
 $$
 Thus, if $C\subset N\setminus S$:
 \begin{multline*}
 \xi^g ({C\cup i};N\setminus j, u_{S}\vert_{N\setminus j})-\xi^g (C;N\setminus j, u_{S}\vert_{N\setminus j}) =
  \\[3mm]
  \frac{1}{s}+ \sum_{k=1}^{n-s-1} \alpha^{n-k} - \sum_{k=1}^{c} \alpha^{n-k} =
 \xi^g ({C\cup i} ;N, u_{S})-\xi^g ({C\cup j};N, u_{S})
 \end{multline*}
 If $C\cap S\neq \emptyset$, then
 \begin{multline*}
 \xi^g ({C\cup i};N\setminus j, u_{S}\vert_{N\setminus j})-\xi^g ({C};N\setminus j, u_{S}\vert_{N\setminus j}) =
 \\[3mm]
 \frac{1}{s-\vert S\cap C\vert}+ \sum_{k=1}^{n-s-1} \alpha^{n-k} - \Bigl (\frac{1}{s-\vert S\cap C\vert +1}+  \sum_{k=1}^{n-s-1} \alpha^{n-k} \Bigr ) =
  \xi^g ({C\cup i};N, u_{S})-\xi^g ({C\cup j};N, u_{S})
 \end{multline*}
 \item
 If $i,j\in N\setminus S$, then $C\cup i\subset N\setminus S$ if, and only if, $C\cup j\subset N\setminus S$ and, therefore condition \eqref{CBC-corta} can be easily checked.
 \end{enumerate}
Now, since $G$-CBC condition is linear $\xi^g$ satisfies it over $\bigcup_{n\geq 1} {\cal G}_n$

\vspace*{5mm}
\noindent \textbf{$G$-linearity}. Let $\xi^g$ be another group value which is defined in the following way for each $N\in {\cal N}$, and every $v\in {\cal G}_N$.

\begin{itemize}
\item[$(A_1)$] If there is at least a null player in $N$, or $(N,v)$ is the unanimity game with respect to the grand coalition $N$, then $\xi^g(C;N,v)=\phi^g(C;N,v)$ for every group $C$ in $N$.
\item[$(A_2)$] Otherwise, $\xi^g(C;N,v)=\phi^g (C;N,v)+k$, being $k\neq 0$ a fixed constant.
\end{itemize}
It is easily checked that $G$-null-player and $G$-SPB hold for $\xi^g$. We will check that the property of coalitional balanced contributions $G$-BMC also holds for this value. Note that all differences in \eqref{CBC-corta} match $\xi^g (D;L,w)$ with $\xi^g (D';L,w)$, for some coalition $L\in \{ N,N\setminus i, N\setminus j\}$ and some game $w\in \{ v,v_{-i},v_{-j}\}$.

Taking into account that $\phi^g=\xi^g$ in case $(N,v)$ is some of the games in the first case $(A_1)$, and in the second one the $k's$ cancel, it holds:
\begin{align*}
 \xi^g ({C\cup i};N,v)- \xi^g({C\cup j} ;N,v) &=  \phi^g ({C\cup i};N,v)- \phi^g ({C\cup j};N,v),
\\[3mm]
 \xi^g ({C\cup i};N\setminus j,v_{-j})- \xi^g (C;N\setminus j,v_{-j}) &=  \phi^g ({C\cup i};N\setminus j,v_{-j})- \phi^g ({C};N\setminus j,v_{-j}),
 \\[3mm]
 \xi^g ({C\cup j};N\setminus i,v_{-i})- \xi^g ({C};N\setminus i,v_{-i}) &=  \phi^g ({C\cup j};N\setminus i,v_{-i})- \phi^g ({C};N\setminus i,v_{-i}),
 \end{align*}
$G$-CBC property holds for $\phi^g$, and so the property does so for $\xi^g$. Trivially, $\xi^g$ fails to verify $G$-linearity and so we are done.


\vspace*{5mm}
\noindent \textbf{$G$-CBC}. Define a group value $\xi^g$ in the following manner. Let $N$ be any finite set of players $N\in {\cal N}$, and let $v\in {\cal G}_N$. If $(N,v)$ is a null game, $\xi^g(C;N,v)$=0 for every $C\in N$. Given a non-null unanimity game $(N,u_S)$ with $S\subsetneq N$, define $\xi^g(C;N,u_S)$ as
\begin{equation}
\xi^g(C;N,u_S)=\begin{cases} 0 & \text{ if $C\cap S=\emptyset$,} \\
 1, & \text{ if $S\subset C$,}
\\
|S\cap C|\times |S\backslash C| & \text{ if $C\cap S\neq \emptyset$ and  $C\cap S\neq S,$}
\end{cases}
\end{equation}
 $\xi^g(C;N,u_N)=\frac{1}{n-c+1}$, for all $C\subset N$, and then extend the value by linearity to ${\cal G}_N$.

Observe that all properties but $G$-CBC hold. The unique one that needs a bit of discussion is the $G$-null player property, which holds when considering the base of unanimity games because in $u_S$ the null players are precisely the players outside $S$, and therefore:
\begin{align*}
\xi^g ({C\cup i};N,u_S) & = 0= \xi^g ({C};N,u_S), \text{ if $S\cap C=\emptyset,$}
\\
\xi^g ({C\cup i};N,u_S) & = 1= \xi^g ({C};N,u_S), \text{ if $S\subset C$, and}
\\
\xi^g ({C\cup i};N,u_S) & = |S\cap (C\cup i)|\times |S\backslash (C\cup i)|=|S\cap C|\times |S\backslash C|= \xi^g ({C};N,u_S), \text{ otherwise,}
\end{align*}
for every player $i\notin S$. Then, taking into account that the Harsanyi dividend $c_S (N,v)$ of any coalition $S$ containing null players in the game $(N,v)$ equals zero, $G$-null player property holds for any $v\in {\cal G}_N$, for all $N\in {\cal N}$.

Let us check by means of a concrete example that the $G$-CBC property fails in this case. Consider $(N,u_S)$ with $N=\{1,2,3\}$, and $S=\{1,2\}$. In the notation of the property, take $C=\{1\}$, $i=2$, and $j=3$. Then:
$$
\begin{array}{c}
\xi^g({\{1,2\}};N,u_S)  =1, \; \xi^g({\{1,3\}};N,u_S)=1\times 1=1,
 \\[3mm]
 \xi^g({\{1,2\}};N\backslash 3, u_S\vert_{N \backslash 3})=1, \; \xi^g({1};N\backslash 3, u_S\vert_{N \backslash 3})=1/2,
 \end{array}
$$
and $\xi^g({\{1,3\}};N\backslash 2, u_S\vert_{N \backslash 2})=0=\xi^g({1};N\backslash 2, u_S\vert_{N \backslash 2})$, because the game $(N\backslash 2, u_S\vert_{N \backslash 2})$ is null. It is clear now that the two sides of the equalities that define the property do not coincide in this case.


\vspace*{5mm}
\noindent \textbf{$G$-symmetry over pure bargaining games}. Consider the additive Shapley value ${\cal A}\phi^g(C;N,v):=\sum_{i\in C} \phi_i(N,v)$.
Observe that all properties but $G$-SPB hold and, clearly, $G$-SPB fails, so we are done.
\blok

\end{document}